\documentclass[a4paper,11pt]{amsart}

\usepackage{amssymb}
\usepackage{amscd}
\usepackage[all]{xy}

\def\frak{\mathfrak}
\def\Bbb{\mathbb}
\def\Cal{\mathcal}

\newtheorem*{prop*}{Proposition}

\newtheorem*{thm*}{Theorem}

\newtheorem*{lem*}{Lemma}

\newtheorem*{cor*}{Corollary}

\newcommand{\id}{\operatorname{id}}

\renewcommand{\o}{\circ}

\let\ccdot\cdot
\def\cdot{\hbox to 2.5pt{\hss$\ccdot$\hss}}

\newcommand{\lpl}
{\mbox{$
\begin{picture}(12.7,8)(-.5,-1)
\put(2,0.2){$+$}
\put(6.2,2.8){\oval(8,8)[l]}
\end{picture}$}}

\newcommand{\al}{\alpha}
\newcommand{\be}{\beta}

\newcommand{\la}{\lambda}

\renewcommand{\phi}{\varphi}
\newcommand{\ph}{\varphi}

\newcommand{\si}{\sigma}

\newcommand{\La}{\Lambda}
\newcommand{\Ga}{\Gamma}
\newcommand{\Ph}{\Phi}
\newcommand{\Ps}{\Psi}
\newcommand{\Om}{\Omega}

\newcommand{\fg}{{\frak g}}

\newcommand{\Rho}{{\mbox{\sf P}}}

\newcommand{\ce}{{\Cal E}}

\def\sideremark#1{\ifvmode\leavevmode\fi\vadjust{\vbox to0pt{\vss
 \hbox to 0pt{\hskip\hsize\hskip1em
 \vbox{\hsize3cm\tiny\raggedright\pretolerance10000
 \noindent #1\hfill}\hss}\vbox to8pt{\vfil}\vss}}}%
\begin{document}
\title[Invariant operators via curved Casimirs:examples]{Conformally Invariant Operators\\ 
via Curved Casimirs: Examples}
\author{Andreas \v Cap, A.\ Rod Gover, and V. Sou\v cek}
\address{A.C.: Fakult\"at f\"ur Mathematik, Universit\"at Wien,
  Nordbergstr. 15, 1090 Wien, Austria\newline
\indent A.R.G.: Department of
  Mathematics, The University of Auckland, Auckland, New Zealand
\newline\indent V.S.: Mathematical Institute, Charles University,
  Sokolovsk\'a 83, Praha, Czech Republic}
\email{Andreas.Cap@esi.ac.at, r.gover@math.auckland.ac.nz,\newline\indent
  soucek@karlin.mff.cuni.cz}  
\dedicatory{Dedicated to Professor J.J.~Kohn on the occasion of his
  75th birthday}
 \keywords{conformally invariant differential operators, curved Casimir operator, GJMS operators}
\begin{abstract}
  We discuss a scheme for a construction of linear conformally
  invariant differential operators from curved Casimir operators; we
  then explicitly carry this out for several examples. Apart from
  demonstrating the efficacy of the approach via curved Casimirs, this
  shows that this method is general in that it applies both in regular
  and in singular infinitesimal character, and also that it can be
  used to construct standard as well as non--standard operators.
  (Nevertheless the scheme discussed here does not recover all
  operators.)  The examples treated include conformally invariant
  operators with leading term, in one case, a square of the Laplacian,
  and in another case, a cube of the Laplacian.
\end{abstract}
 \maketitle

\section{Introduction}
Curved Casimir operators were originally introduced in
\cite{Cap-Soucek} in the setting of general parabolic geometries. For
any natural vector bundle associated to such a geometry, there is a
curved Casimir operator which acts on the space of smooth sections of
the bundle. The name of the operator is due to the fact that on the
homogeneous model of the geometry, it reduces to the canonical action
of the quadratic Casimir element. The curved Casimir operators may be
expressed by a simple (Laplacian like) formula in terms of the
fundamental derivative from \cite{tams} and hence share the very
strong naturality properties of the fundamental derivative.  While on
a general natural vector bundle the curved Casimir operator is of
order at most one, it always acts by a scalar on a bundle associated
to an irreducible representation. This scalar can be easily computed
from representation theory data. It was already shown in
\cite{Cap-Soucek} that using this and the naturality properties, one
can use the curved Casimir operators systematically to construct
higher order invariant differential operators. Namely,
\cite{Cap-Soucek} contains a general construction of splitting
operators, which are basic ingredients in all versions of the curved
translation principle.

Essentially the same construction can be also used to directly obtain
invariant differential operators acting between sections of bundles
associated to irreducible representations. One considers the tensor
product of a tractor bundle and an irreducible bundle. Such a bundle
has an invariant filtration such that the quotients of subsequent
filtrations components are completely reducible.  Adapting the action
of the centre of the structure group (which amounts to tensoring with
a density bundle), one may force a coincidence of curved Casimir
eigenvalues for irreducible components in different subquotients.  As
we shall see this leads to an invariant linear differential operator
acting between the sections of these components. A more difficult
issue is to prove, in some general context, that the resulting
operator is nontrivial. General tools for doing this systematically
are developed in \cite{Cap-Gover}.

The purpose of this article is to carry out the construction of
invariant operators explicitly for a few examples in the realm of
conformal structures. First, this shows that the general ideas can be
made explicit rather easily.  Secondly, it shows that the curved
Casimir operators can be used to produce both standard and
non--standard operators, and they work both in regular and in singular
infinitesimal character; this is in contrast to the usual
constructions of BGG sequences as developed in \cite{CSS,CD}.

Finally, we want to indicate how some of the
well known and intriguing phenomena concerning conformally invariant
powers of the Laplacian show up in the approach via curved Casimirs.
In particular, this concerns the fact that the critical powers of the
Laplacian are not strongly invariant and the non--existence of
supercritical powers of the Laplacian.

\subsection*{Acknowledgements} The basic ideas for this work evolved
during meetings of the first and second author in the ``Research in
Teams'' programme ``Conformal and CR geometry: Spectral and nonlocal
aspects'' of the Banff International Research Station (BIRS) and of
the first and third author at the Erwing Schr\"odinger Institute (ESI)
in Vienna. It was completed during a meeting of all three authors at
the New Zealand Institute of Mathematics and Its Applications (NZIMA)
thematic programme ``Conformal Geometry and its Applications'' hosted
by the Departement of Mathematics of the University of Auckland.

The authors were supported finacially as follows: AC by project P
19500--N13 of the ``Fonds zur F\"or\-der\-ung der wissenschaftlichen
Forschung'' (FWF); ARG by Marsden Grant no. 06-UOA-029; VS by the
institutional grant MSM 0021620839 and by the grant GA\v CR
201/08/0397. Finally we thank the referee for helpful comments.

\section{Examples of conformally invariant operators\\
constructed from curved Casimirs}\label{2}

\subsection{Conformal structures, tractor bundles, and tractor
  connections}\label{2.1} 

We shall use the conventions on conformal structures from
\cite{confamb}. We consider a smooth manifold $M$ of dimension $n\geq
3$ endowed with a conformal equivalence class $[g]$ of
pseudo--Riemannian metrics of some fixed signature $(p,q)$. We use
Penrose abstract index notation, so $\ce^a$ will denote the tangent
bundle $TM$ and $\ce_a$ the cotangent bundle $T^*M$. Several upper or
lower indices will indicate tensor products of these basic bundles,
round brackets will denote symmetrisation,   square brackets
alternation, and the subscript $0$ indicates a tracefree part.

For $w\in\Bbb R$ we denote by $\ce[w]$ the bundle of
$(-\tfrac{w}{n})$--densities on $M$. For any choice of metric $g$ in
the conformal class, sections of $\ce[w]$ can be identified with
smooth functions but changing from $g$ to $\hat g=f^2g$ (where $f$ is a
positive smooth function on $M$), this function changes by
multiplication by $f^w$. Adding $[w]$ to the notation for a bundle
indicates a tensor product by $\ce[w]$. Using these conventions, the
conformal structure can be considered as a smooth section
$\mathbf{g}_{ab}$ of the bundle $\ce_{(ab)}[2]$, called the
\textit{conformal metric}. Contraction with $\mathbf{g}_{ab}$ defines
an isomorphism $\ce^a\cong\ce_a[2]$, whose inverse can be viewed as a
smooth section $\mathbf{g}^{ab}$ of $\ce^{(ab)}[-2]$. We shall use
$\mathbf{g}_{ab}$ and $\mathbf{g}^{ab}$ to raise and lower tensor
indices.

The \textit{standard tractor bundle} of $(M,[g])$ will be denoted by
$\ce^A$. This is a vector bundle of rank $n+2$ canonically associated
to the conformal structure. It is endowed with a canonical bundle
metric $h_{AB}$ of signature $(p+1,q+1)$ which will be used to raise
and lower tractor indices. Further, there is a canonical linear
connection $\nabla^{\Cal T}$ on $\ce^A$ which is equivalent to the
conformal Cartan connection. Finally, there is a canonical inclusion
$\ce[-1]\hookrightarrow \ce^A$ whose image is an isotropic line
subbundle of $\ce^A$. This can be viewed as a canonical section $X^A$
of $\ce^A[1]$ which satisfies $h_{AB}X^AX^B=0$. Next, $X_A:=h_{AB}X^B$
can be interpreted as a projection $\ce^A\to\ce[1]$. These data fit
together to define a composition series for $\ce^A$ that we shall
denote $\ce[1]\lpl \ce_a[1]\lpl \ce[-1]$; the second $\lpl$ indicates
that $\ce[-1]$ is a subbundle of $\ce^A$ while the first $\lpl$ means
$\ce_a[1]$ is (isomorphic to) a subbundle of the quotient bundle 
$\ce^A/ \ce[-1] $ and that $(\ce^A/ \ce[-1])/\ce_a[1]\cong\ce[1]$. 
(The motivation for the notation is that summands include, while there 
is a projection onto direct summands). General tractor bundles then correspond to
$SO(p+1,q+1)$--invariant subspaces in tensor powers of $\Bbb
R^{(p+1,q+1)}$, and we will also use abstract index notation for
tractor indices.

Any choice of a metric $g$ in the conformal class gives rise to a
splitting $\ce^A\cong \ce[1]\oplus\ce_a[1]\oplus\ce[-1]$ of the
composition series. The change of this splitting caused by a conformal
rescaling of the metric can be easily described explicitly, see
\cite{luminy}, but we will not need these formulae here. What we will
need is the expression of the tractor connection in the splitting
associated to $g$ in terms of the Levi--Civita connection $\nabla$ of
$g$. To formulate this efficiently, we need the \textit{adjoint
  tractor bundle} of $(M,[g])$. By definition, this is the bundle
$\frak{so}(\ce^A)\cong\ce_{[AB]}$ of endomorphisms of $\ce^A$ which
are skew symmetric with respect to the tractor metric. By definition,
this bundle naturally acts on $\ce^A$ and hence (tensorially) on any
tractor bundle. 

Now the composition series of $\ce^A$ gives rise to a composition
series $\ce_{[AB]}=\ce^a \lpl(\ce_{[ab]}[2]\oplus\ce[0])\lpl\ce_a$, so
the adjoint tractor bundle contains $T^*M$ as a natural subbundle and
has $TM$ as a natural quotient. A choice of metric in the conformal
class also splits this composition series, so we obtain an isomorphism
$\ce_{[AB]}\cong\ce^a\oplus(\ce_{[ab]}[2]\oplus\ce[0])\oplus\ce_a$
depending on the choice of metric. In particular, we can view elements
of $T^*M$ naturally as elements of the adjoint tractor bundle and,
choosing a metric in the conformal class, we can also view elements of
$TM$ as elements in the adjoint tractor bundle.

There are explicit formulae how the identifications of tractor bundles
behave under a conformal change of metric, see e.g.~Theorem 1.3 of
\cite{luminy}. However, we will not need this formulae here, since we
will always deal with operations which are known to be invariant in
advance and use the splittings only to compute explicit formulae for
these operations. We shall only need the formula for the canonical
tractor connection in a splitting, which also can be found in Theorem
1.3 of \cite{luminy}. This formula is given in the proposition below.
Note that, comparing with \cite{luminy}, the difference in the sign of
the term involving the Rho tensor (also sometimes called the Schouten
tensor) is due to the fact that \cite{luminy} uses a different sign
convention for the Rho--tensor than \cite{confamb}.

\begin{prop*}
  Consider a tractor bundle $\Cal T\to M$ for a conformal structure
  $[g]$ on $M$, and let $\nabla^{\Cal T}$ be the canonical tractor
  connection on $\Cal T$. Choose a metric $g$ in the conformal class
  with Rho tensor $\Rho$ and let $\nabla$ be its Levi Civita
  connection, acting on $\Cal T$ via the isomorphism with a direct sum
  of weighted tensor bundles induced by the choice of metric. Further
  let us denote by $\bullet$ both the actions of $T^*M$ and of $TM$
  (the latter depending on the choice of metric) coming from the
  inclusion of the bundles into the adjoint tractor bundle. Then for
  any vector field $\xi\in\frak X(M)$ and any section $s\in\Ga(\Cal
  T)$ we have
$$
\nabla^{\Cal T}_\xi s=\nabla_\xi s+\xi\bullet s-\Rho(\xi)\bullet s. 
$$ 
\end{prop*}

\subsection{A formula for the curved Casimir operator}\label{2.2}
The main tool used to efficiently treat examples is a new formula for
the curved Casimir operator acting on the tensor product of a tractor
bundle and an irreducible bundle. Consider the group $G:=SO(p+1,q+1)$
and let $P\subset G$ be the stabiliser of an oriented isotropic line
in the standard representation $\Bbb R^{(p+1,q+1)}$ of $G$. Then it is
well known that $P$ is the semidirect product of the (orientation
preserving) conformal group $CSO(p,q)$ and a normal vector subgroup
$P_+\cong\Bbb R^{n*}$. It is also well known that a conformal
structure of signature $(p,q)$ on a smooth manifold $M$ determines a
canonical Cartan geometry of type $(G,P)$, so in particular there is a
canonical principal bundle on $M$ with structure group $P$. Forming
associated bundles, any representation of the group $P$ gives rise to
a natural vector bundle on conformal manifolds.

The conformal group $CSO(p,q)$ is naturally a quotient of $P$, so any
representation of $CSO(p,q)$ gives rise to a representation of $P$.
The resulting representations turn out to be exactly those
representations of $P$ which are completely reducible, so they split
into direct sums of irreducibles. The corresponding bundles are called
\textit{completely reducible bundles} and they split into direct sums
of \textit{irreducible bundles}. The completely reducible bundles are
exactly the usual tensor and density bundles. On the other hand, one
can look at restrictions to $P$ of representations of $G$, and these
give rise to tractor bundles. The standard tractor bundle $\ce^A$ and
the adjoint tractor bundle $\ce_{[AB]}$ from \ref{2.1} above correspond
to the standard representation $\Bbb R^{(p+1,q+1)}$ respectively the
adjoint representation $\frak{so}(p+1,q+1)$ of $G$ in this way.

Now recall first from Theorem 3.4 of \cite{Cap-Soucek} that the curved
Casimir operator on an irreducible bundle $W\to M$ acts by a real
multiple of the identity, and we denote the corresponding scalar by
$\be_W$. This scalar can be computed in terms of weights of the
representation which induces $W$. If the lowest weight of this
representation is $-\nu$, then $\be_W=\langle\nu,\nu+2\rho\rangle$,
where $\rho$ is half the sum of all positive roots. On a completely
reducible bundle, the action of the curved Casimir is tensorial and
can be obtained by decomposing the bundle into irreducible pieces, multiplying
each piece by the corresponding factor and then adding back up.

\begin{prop*}
  Let $(M,[g])$ be a conformal manifold of signature $(p,q)$ and let
  $\Cal T\to M$ be a bundle which can be written as the tensor product
  of a tractor bundle and an irreducible bundle. Choose a metric $g$
  in the conformal class and let $\nabla$ be its Levi--Civita
  connection, acting on $\Cal T$ via the identification with a
  completely reducible bundle induced by the choice of $g$. Further,
  let $\be:\Cal T\to \Cal T$ be the bundle map which, in this
  identification, acts on each irreducible component $W\subset\Cal T$
  by multiplication by $\be_W$. Let $\bullet$ denote the action of
  $T^*M$ on $\Cal T$ coming from the natural action on the tractor bundle.
  Then for a local orthonormal frame $\xi_\ell$ for $TM$ with dual
  frame $\ph^\ell$ for $T^*M$, the curved Casimir operator $\Cal C$
  acts on $s\in\Ga(\Cal T)$ by
  $$
  \Cal C(s)=\be(s)-2\textstyle\sum_{\ell}\ph^\ell\bullet
  (\nabla_{\xi_\ell}s-\Rho(\xi_\ell)\bullet s)
$$
\end{prop*}
\begin{proof}
  We use the formula for $\Cal C$ in terms of an adapted local frame
  for the adjoint tractor bundle from Proposition 3.3 of
  \cite{Cap-Soucek}. Having chosen the metric $g$, the adjoint tractor
  bundle splits as $TM\oplus \frak{so}(TM)\oplus T^*M$, and for any
  local frame $\{A_r\}$ for $\frak{so}(TM)$, the local frame
  $\{\xi_\ell,A_r,\ph^\ell\}$ for the adjoint tractor bundle is
  evidently adapted. According to Proposition 3.3 of
  \cite{Cap-Soucek}, one may write $\Cal C(s)$ as the sum of
  $-2\sum_\ell\ph^\ell\bullet D_{\xi_\ell}s$ (with $D$ denoting the
  fundamental derivative) and a tensorial term, in
  which only actions of elements of $\frak{so}(TM)$ show up. Hence the
  latter term preserves any irreducible summand of $\Cal T$, and the
  proof of Theorem 3.4 of \cite{Cap-Soucek} shows that, on such a
  summand $W$, $\Cal C(s)$ acts by multiplication by $\be_W$. To complete the
  proof, it thus suffices to show that
$$
D_{\xi_\ell}s=\nabla_{\xi_\ell}s-\Rho(\xi_\ell)\bullet s.
$$
If $\Cal T$ is a tractor bundle, then this follows immediately from
the formula for the fundamental derivative in section 1.7 of
\cite{luminy}. The formula there (applied to standard tractors) shows
that $D_{\xi_\ell}$ equals $\nabla_{\xi_\ell}$ on the tangent bundle
and on a non--trivial density bundle. By naturality, this is true for
arbitrary irreducible bundles, and the result follows. 
\end{proof}

This formula shows that to compute  explicitly the curved Casimir on
the tensor product of a tractor bundle with an irreducible bundle, only
two ingredients are needed:  first we need to systematically
compute the numbers $\be_W$, and second  we need an
explicit formula for the action of $T^*M$ on the tractor bundle, since
this can be first used to compute $\Rho(\xi)\bullet s$ and then the
action of $\ph^\ell$. 

\subsection{The construction principle}\label{2.2a}
The construction principle we use is actually very close to the
construction of splitting operators in section 3.5 of
\cite{Cap-Soucek}. Let $\Cal T$ be the tensor product of a tractor
bundle and a tensor bundle. The natural filtration of the tractor
bundle (inherited from the filtration of the standard tractor bundle
from \ref{2.1}) induces a natural filtration of $\Cal T$, which we
write as $\Cal T=\Cal T^0\supset\Cal T^1\supset\dots\supset\Cal
T^N$. Each of the subquotients $\Cal T^i/\Cal T^{i+1}$ splits into a
direct sum of irreducible tensor bundles. On sections of each of these
bundles, the curved Casimir operator acts by a scalar by Theorem 3.4
of \cite{Cap-Soucek}, and this scalar is computable from the highest
(or lowest) weight of the inducing representation. We denote by
$\be_i^1,\dots,\be_i^{n_i}$ the different scalars that occur in this
way. 

Now define $L_i:=\prod_{\ell=1}^{n_i}(\Cal C-\be_i^\ell)$. This can be
viewed as a differential operator of order $\leq n_i$ acting on
sections of $\Cal T$. Moreover, naturality of the curved Casimir
operator implies that $L_i$ preserves each of the subspaces formed by
sections of one filtration component. Moreover, for each $j$, the
operator induced on sections of $\Cal T^j/\Cal T^{j+1}$ is given by
the same formula, but with $\Cal C$ being the curved Casimir operator
for that quotient bundle. In particular, this implies that $L_i$
induces the zero operator on $\Ga(\Cal T^i/\Cal T^{i+1})$ and hence
$L_i(\Ga(\Cal T^i))\subset\Ga(\Cal T^{i+1})$. 

Now fix indices $i<j$ and an irreducible component $W\subset \Cal
T^i/\Cal T^{i+1}$. Consider the composition $\pi_j\o L_j\o\dots\o
L_{i+1}$, where $\pi_j$ is the tensorial operator induced by the
projection $\Cal T^i\to\Cal T^i/\Cal T^{j+1}$. Evidently, this
composition defines a differential operator mapping sections of $\Cal
T^i$ to sections of $\Cal T^i/\Cal T^{j+1}$. However, by construction,
sections of $\Cal T^{i+1}$ are mapped to sections of $\Cal T^{i+2}$ by
$L_{i+1}$, which are mapped to sections of $\Cal T^{i+3}$ by
$L_{i+2}$, and so on. Hence our operator factors to sections of $\Cal
T^i/\Cal T^{i+1}$ and restricting to sections of $W$, we obtain an
operator $L:\Ga(W)\to\Ga(\Cal T^i/\Cal T^{j+1})$.

In section 3.5 of \cite{Cap-Soucek}, it is then assumed that the
Casimir eigenvalue $\be$ corresponding to the irreducible  bundle $W$
is different from all the $\be^k_\ell$ for $i<k\leq j$ and all
$\ell$. In that case, composing the projection $\Cal T^i/\Cal
T^j\to\Cal T^i/\Cal T^{i+1}$ with $L$, one obtains a non--zero
multiple of the identity, and hence $L$ is a splitting operator. 

But now let us assume that (with appropriate numeration)
$\be=\be_j^1$, and let $\tilde W\subset\Cal T^j/\Cal T^{j+1}$ be the
sum of the irreducible components corresponding to this eigenvalue.
Then we can write $L_j$ as $(\Cal C-\be)\o\tilde L_j$ where operator 
$\tilde L_j$ is a  polynomial in $\Cal C $.
Next, since all
polynomials in $\Cal C$ commute, we can also write the composition
$\pi_j\o L_j\o\dots\o L_{i+1}$ as $\pi_j\o\tilde L_j\o\dots\o L_{i+1}\o
(\Cal C-\be)$. But the latter composition evidently maps a section of
$\Cal T^i$, whose image in $\Cal T^i/\Cal T^{i+1}$ has values in $W$
to a section of $\Cal T^j$. Hence in this case, $L$ has values in
sections of $\Cal T^j/\Cal T^{j+1}$. Moreover, since 
$$
(\Cal C-\be)\o \pi_j\o L_j\o\dots\o L_{i+1}=\pi_j\o L_j\o\dots\o
L_{i+1}\o (\Cal C-\be)
$$
evidently induces the zero operator on $\Ga(W)$, we conclude that
$L$ actually has values in $\Ga(\tilde W)$, so we have obtained an
operator $L:\Ga(W)\to\Ga(\tilde W)$.

\subsection{Computing the Casimir eigenvalues}\label{2.3}
We need a systematic notation for weights and their relation to
irreducible bundles. Since these issues are slightly different in
even and odd dimensions, we will restrict our attention to the case of
even dimension $n=2m$ from now on; in many senses  conformally invariant powers
of the Laplacian are  more interesting in even dimensions. Note
that the weights involved are actually defined on the complexification
$\fg_{\Bbb C}=\frak{so}(2m+2,\Bbb C)$ of $\fg=\frak{so}(p+1,q+1)$. The
process of assigning weights to real representations of $\fg$ and
$\fg_0=\frak{co}(p,q)$ is discussed in section 3.4 of
\cite{Cap-Soucek}.

We use the notation from chapter 19 of \cite{Fulton-Harris} for
weights for $\fg_{\Bbb C}=\frak{so}(2m+2,\Bbb C)$. Hence weights will
be denoted by tuples $(a_1,a_2,\dots,a_{m+1})$, and the (highest
weights of) irreducible tensor representations (we will not require
any spin representations) correspond to tuples in which all the $a_i$
are integers and $a_1\geq a_2\geq\dots\geq a_{n-1}\geq \pm a_n$. For
example, for $i<m$, the $i$th exterior power $\La^i\Bbb C^{2m+2}$ is
irreducible and corresponds to the tuple $a_1=\dots=a_i=1$ and
$a_{i+1}=\dots=a_{m+1}=0$. In this notation, the half sum of all
positive roots is given by $\rho=(m,m-1,\dots,1,0)$.

Weights for the complexification of $\fg_0$ can be viewed as
functionals on the same space, the conditions on dominance and
integrality are different, however. Since this difference concerns
the first entry only, we use the notation $(a_1|a_2,\dots,a_{m+1})$
for these weights.

The formula for the Casimir eigenvalues is in terms of lowest weights.
For weights of tensor representations of $\fg_{\Bbb C}$ this coincides
with the highest weight since any such representation is isomorphic to
its dual. It will be helpful to keep in mind that the lowest weight of
a representation of $\fg_{\Bbb C}$ coincides with the lowest weight of
the irreducible quotient representation of $(\fg_0)_{\Bbb C}$. This is
sufficient to understand the correspondence between weights and
irreducible bundles. For example, the standard representation of
$\fg_{\Bbb C}$ corresponds to the weight $(1,0,\dots,0)$ and the
standard tractor bundle $\ce^A$, whose irreducible quotient is
$\ce[1]$. Hence $\ce[1]$ corresponds to the weight $(1|0,\dots,0)$ and
therefore $\ce[w]$ corresponds to $(w|0,\dots,0)$ for $w\in\Bbb R$.

More generally, for $i<m$, the $i$th exterior power of the standard
representation corresponds to $(1,\dots,1,0,\dots,0)$ (with $i$
entries equal to $1$) and is also a notation for $\La^i\ce^A$, which
clearly has $\La^{i-1}\ce_a\otimes\ce[i]$ as an irreducible quotient.
Hence $\ce_a$ and $\ce^a$ correspond to $(-1|1,0,\dots,0)$ and
$(1|1,0,\dots,0)$, respectively, and $\ce_{[ab]}[w]$ corresponds to
$(w-2|1,1,0\dots,0)$.  The highest weight of $S^k_0\ce_a$ is just $k$
times the highest weight of $\ce_a$, so $S^k_0\ce_a[w]$ corresponds to
$(w-k|k,0,\dots,0)$, and so on.

The final ingredient needed to apply the formula for Casimir
eigenvalues is the inner product on weights. Taking as our invariant
bilinear form half the trace form on the Lie algebra (which leads to
the nicest conventions), one simply obtains the standard inner
product. For example, for $W=S^k_0\ce_a[w]$ the corresponding weight
$\la=(w-k|k,0,\dots,0)$ and 
$$
\be_W=\langle \la,\la+2\rho\rangle=(w-k)(w+2m-k)+k(2m+k-2).
$$

\subsection{Standard tractors twisted by one--forms}\label{2.4}
We now have all the technical input at hand, so we look at the first
example. Consider the tensor product $\ce_a[w]\otimes\ce^A$ of the
standard tractor bundle with the bundle of weighted one--forms. We
will describe the curved Casimir operator on this bundle and find
basic splitting operators and all the invariant differential operators
between irreducible bundles that can be constructed from this curved
Casimir. From the composition series for $\ce^A$ from \ref{2.1} we get
a composition series $\ce_a[w+1]\lpl \ce_{ab}[w+1]\lpl\ce_a[w-1]$ for
our bundle. We use the convention that in the middle slot the first
indices come from $\ce_a[w]$ and the second ones from the tractor
bundle. The middle term decomposes as $\ce_{(ab)_0}[w+1]\oplus
\ce[w-1]\oplus \ce_{[ab]}[w+1]$, and if $n\geq 6$ then each of the
summands is irreducible. For $n=4$, the bundle $\ce_{[ab]}[w+1]$
splits into the sum of self--dual and anti--self--dual two forms,
which then are irreducible. As we shall see below, however, this does
not cause any change, so we can treat all even dimensions $\geq 4$
uniformly. According to these decompositions, sections
$\ce_a[w]\otimes\ce^A$ will be written as vectors of the form
$$
\begin{pmatrix}
  \si_a\\ A_{ab}\quad |\quad\al\quad |\quad B_{ab}\\\rho_a
\end{pmatrix}
$$
with $A_{ab}=A_{(ab)_0}$ and $B_{ab}=B_{[ab]}$. Following the usual
conventions the top slot is the projecting slot, so $\si_a$ has weight
$w+1$ while $\rho_a$ has weight $w-1$. 
The action of $\ph^i\in\Om^1(M)$
on the standard tractor bundle can be immediately computed from the
matrix representation of $\fg$, and using this, we obtain
$$
\ph_i\cdot
\begin{pmatrix}
  \si_a\\ A_{ab}\quad |\quad\al\quad |\quad B_{ab}\\\rho_a
\end{pmatrix}=
\begin{pmatrix}
  0\\ -\si_{(a}\ph_{b)_0}\quad |\quad -\si^i\ph_i\quad
  |\quad -\si_{[a}\ph_{b]}\\
  A_{ab}\ph^b+\tfrac{1}{n}\al\ph_a+B_{ab}\ph^b 
\end{pmatrix}.
$$
The Casimir eigenvalues $\be_W$ for the irreducible components in
our bundle can be computed using the formulae from \ref{2.3}. In
dimension four, the self--dual and anti--self--dual parts in
$\ce_{[ab]}[w+1]$ correspond to the weights $(w-1|1,1)$ and
$(w-1|1,-1)$, respectively. This shows that, for any choice of the
weight $w$, the curved Casimir operator acts by the same scalar on
sections of the two bundles. Hence in our constructions schemes for
operators we may always treat the sum of these two bundles as if it
were a single irreducible component, which shows that the general
discussion applies to dimension four as well. The numbers $\be_W$ are
given by
\begin{equation}
  \label{Cas-ew-1}
  \begin{pmatrix}
  a_0+n-1 \\ a_0-2w+n+1\quad |\quad a_0-2w-n+1\quad |\quad
  a_0-2w+n-3\\ a_0-4w-n+3
\end{pmatrix},
\end{equation}
where $a_0=w(w+n)$.  We will denote the eigenvalue in the top slot by
$\be_0$, the one in the bottom slot by $\be_2$, and the three middle
ones by $\be_1^1$, $\be_1^2$ and $\be_1^3$. Using this, we can now
write out the curved Casimir operator explicitly. Acting by $\nabla
-\Rho\bullet\ $ on a typical element, we get
$$
\begin{pmatrix}
\nabla_a\si_b \\ \nabla_aA_{bc}+\Rho_{a(b}\si_{c)_0}\quad |\quad
\nabla_a\al+\Rho_a{}^d\si_d\quad |\quad \nabla_aB_{bc}-\Rho_{a[b}\si_{c]}\\ 
\nabla_a\rho_b-\Rho_a{}^dA_{db}-\tfrac{1}{n}\al\Rho_{ab}+\Rho_a{}^dB_{db} 
\end{pmatrix}.
$$
Via Proposition \ref{2.2} we can compute $\Cal C$ by applying to this the
action of the index $a$, multiplying the result by $-2$, and adding
the components of the original element multiplied by the appropriate
scalar. This gives 
$$
\begin{pmatrix}
\be_0\si_a \\
\be_1^1A_{ab}+2\nabla_{(a}\si_{b)_0}\quad |\quad
\be_1^2\al+2\nabla^c\si_c \quad |\quad
\be_1^3B_{ab}+2\nabla_{[a}\si_{b]} \\
\be_2\rho_a-2\nabla^cA_{ca}-2\Rho^c{}_{(c}\si_{a)_0}-\tfrac{2}{n}\nabla_a\al-
\tfrac{2}{n}\Rho_a{}^c\si_c-2\nabla^cB_{ca}-2\Rho^c{}_{[c}\si_{a]} 
\end{pmatrix}.
$$
>From this formula, we can immediately read off a number of invariant
first order splitting operators as well as invariant first order
operators between irreducible bundles. For example, elements with
$\si_a=\al=B_{ab}=0$ form a natural subbundle of $\Cal
E_A\otimes\ce_a[w]$ for each $w$. On sections of this natural
subbundle, $\Cal C-\be_2\id$ defines a natural operator given by 
$$
\begin{pmatrix}
  0\\ A_{ab}\quad |\quad 0 \quad |\quad 0\\\rho_a
\end{pmatrix}\mapsto\begin{pmatrix} 0 \\
  (\be_1^1-\be_2)A_{ab} \quad |\quad 0 \quad |\quad 0\\
  -2\nabla^cA_{ca}.
\end{pmatrix}
$$
Since the value is independent of $\rho_a$, it descends to a natural
operator defined on $\ce_{(ab)_0}[w+1]$. If $\be_1^1-\be_2\neq 0$ or
equivalently $w\neq 1-n$, this is the splitting operator
$\Ga(\ce_{(ab)_0}[w+1])\to\Ga(\ce^A_a[w])$ as constructed in
\cite{Cap-Soucek}. However, for $w=1-n$, the operator has values in
the natural subbundle $\ce_a[-n]\subset \ce^A_a[1-n]$, so we obtain a
natural differential operator
$\Ga(\ce_{(ab)_0}[2-n])\to\Ga(\ce_a[-n])$ given by $A_{ab}\mapsto
-2\nabla^bA_{ba}$. This is the adjoint of the conformal Killing
operator. 

In the same way, one obtains splitting operators for the other middle
slots, and first order operators $\ce[0]\to\ce_a[0]$ (the exterior
derivative from functions to one--forms) and
$\ce_{[ab]}[4-n]\to\ce_a[2-n]$ (the divergence or equivalently the
exterior derivative from $(n-2)$--forms to $(n-1)$--forms). 

To construct invariant operators defined on the quotient bundle
$\ce_a[w+1]$, consider the differences of the $\be$'s from $\be_0$,
which are given by
$$
\begin{pmatrix}
  0\\ c_1^1\quad |\quad c_1^2\quad |\quad c_1^3\\ c_2 
\end{pmatrix}:=
\begin{pmatrix}
  0\\ 2w-2\quad |\quad 2w+2n-2\quad |\quad 2w+2\\ 4w+2n-4 
\end{pmatrix}
$$
>From the formula for $\Cal C$ from above, we can read off the three
first order invariant operators obtained in the case that $c_1^i=0$.
For $c_1^1=0$, i.e.~$w=1$ we get the conformal Killing operator
$\ce_a[2]=\ce^a\to \ce_{(ab)_0}[2]$. For $c_1^2=0$ we get $w=1-n$ and
we obtain the divergence $\ce_a[2-n]\to \ce[-n]$ (or equivalently the
exterior derivative from $(n-1)$--forms to $n$--forms. Finally,
$c_1^3=0$ corresponds to $w=-1$ as this gives the exterior derivative from one--forms to
two forms. 

To construct the full splitting operator defined on $\ce_a[w+1]$
respectively an operator from this bundle to $\ce_a[w-1]$ (for a
special value of $w$), we have to form $(\Cal C-\be_2)\o(\Cal
C-\be_1^1)\o(\Cal C-\be_1^2)\o(\Cal C-\be_1^3)$. This gives a
splitting operator provided that all $c_1^i$ and $c_2$ are nonzero by
Theorem 2 of \cite{Cap-Soucek}. For $c_2=0$, i.e.~$w=1-\tfrac{n}{2}$,
we see from \ref{2.2a} that we obtain an invariant differential
operator $\Ga(\ce_a[2-\tfrac{n}{2}])\to\Ga(\ce_a[-\tfrac{n}{2}])$ of
order at most two. We can immediately calculate  this operator using
the above formula for $\Cal C$. Its value on $\si_a$ reads as
$$
\begin{pmatrix}
  c_2c_1^1c_1^2c_1^3\si_a \\ 2 c_2c_1^2c_1^3\nabla_{(a}\si_{b)_0}
  \quad |\quad 2c_2c_1^1c_1^3\nabla^i\si_i\quad |\quad
  -2c_2c_1^1c_1^2\nabla_{[a}\si_{b]}\\ A_a(\si)
\end{pmatrix},
$$
where 
\begin{align*}
  A_a(\si)=& -2c_1^2c_1^3(2\nabla^i\nabla_{(i}\si_{a)_0}+c_1^1
  P^i{}_{(i}\si_{a)_0})-\tfrac2n
  c_1^1c_1^3(2\nabla_a\nabla^i\si_i+c_1^2P_a{}^i\si_i)\\
  +&2c_1^1c_1^2(2\nabla^i\nabla_{[i}\si_{a]}-c_1^3P^i{}_{[a}\si_{i]})
\end{align*}
In particular, we see that for $c_2=0$, only the bottom slot is
non--zero, and, as expected, we obtain an invariant operator
$\si\mapsto A_a(\si)$. We can easily compute the principal part of
this operator by looking only at the second order terms and commuting
derivatives. This shows that, up to a non--zero factor, the principal
part is given by 
$$
\si_a\mapsto (n-2)\big(n\Delta\si_a-4\nabla_a\nabla^i\si_i\big). 
$$
In particular, except for the case $n=2$, which is geometrically
irrelevant, we obtain a true second order operator.

\medskip

Collecting our results, we see that from curved Casimirs on the bundle
$\ce_a[w]\otimes\ce^A$ we obtain seven invariant operators between
irreducible bundles. Six of these are first order, while one is of
order two. The first order operators belong to two different BGG
sequences. The two exterior derivatives and the two divergences are
part of the de--Rham sequence, i.e.~the BGG sequence of the trivial
representation. The conformal Killing operator and its adjoint are
well known to be part of the BGG sequence corresponding to the adjoint
representation. Finally, for $n\geq 6$ the second order operator
$\Ga(\ce_a[2-\tfrac{n}{2}])\to\Ga(\ce_a[-\tfrac{n}{2}])$ is not part
of any BGG sequence, since the corresponding representations (or
rather the Verma modules associated to their duals) have singular
infinitesimal character. Moreover, the resulting operator is a
non--standard operator. Hence we see that even for this simple
example, we obtain both standard and non--standard operators both in
regular and singular infinitesimal character. In dimension four, the
situation is slightly different, since the two critical weights $w=-1$
and $w=1-\tfrac{n}2$ coincide. This means that the second order
operator is obtained as the composition of the divergence and the
exterior derivative. Hence for $n=4$, we obtain the Maxwell operator,
which is a standard operator in the BGG--sequence of the trivial
representation.

\section{Conformally invariant powers of the Laplacian}\label{3}
In this section, we show how to construct the conformally invariant
square and cube of the Laplacian from curved Casimir operators. There
are some well known subtle phenomena concerning these operators. As
shown in \cite{Graham:nonex} in dimension four and in
\cite{Gover-Hirachi:JAMS} in general, there are no conformally
invariant powers of the Laplacian in even dimensions $n=2m$ whose
order exceeds $n$. Moreover, the $m$th power (called the critical
power) is of much more subtle nature than the lower powers. As shown
in \cite{Eastwood-Slovak}, for all lower powers of the Laplacian (as
well as all operators occurring in BGG--sequences) there are formulae
which are strongly invariant (induced from homomorphisms on
semi--holonomic jet modules), while the critical powers do not have
this property. As we shall see, these phenomena are reflected very
nicely in the constructions via curved Casimir operators. For the
square of the Laplacian, a different construction has to be used in
the critical dimension four. On the other hand, the construction for
the cube of the Laplacian completely breaks down in dimension four.

\subsection{The square of the Laplacian in dimensions $\neq 4$}\label{3.1}
We consider the tracefree part in the symmetric square of the standard
tractor bundle twisted by a weight, i.e.~the bundle $\ce^{(AB)_0}[w]$.
>From the composition series of the standard tractor bundle in
\ref{2.1} we see that
$$
\ce^{(AB)_0}[w]=\ce[w+2]\lpl \ce_a[w+2] \lpl
(\ce_{(ab)_0}[w+2]\oplus \ce[w])\lpl \ce_a[w]\lpl \ce[w-2].
$$
We will again use a vector notation with the projecting slot on
top. To compute the action of $\frak p_+$, one has to represent
typical elements in each slot by tensor products of standard tractors,
and then compute the tensorial action. It is obvious how to get such
representatives, except for the two components in the middle. Using $\vee$ to denote the symmetric tensor product, 
the
representatives for $\ce[w]$ are the multiples of the element 
$$
\left(\begin{smallmatrix}1\\0\\0\end{smallmatrix}\right)\vee
\left(\begin{smallmatrix}0\\0\\1\end{smallmatrix}\right)-
\tfrac{1}{n}\textstyle\sum_j \left(\begin{smallmatrix}0 \\ e_j\\
    0\end{smallmatrix}\right)\vee \left(\begin{smallmatrix}0\\ e^j\\
    0\end{smallmatrix}\right) 
$$
for dual bases $\{e_j\}$ and $\{e^j\}$. On the other hand, typical
representatives for the elements in $\ce_{(ab)_0}[w+2]$ are given by
the sum of
$\left(\begin{smallmatrix}0\\\mu_a\\0\end{smallmatrix}\right)\vee
\left(\begin{smallmatrix}0\\\nu_b\\0\end{smallmatrix}\right)$ and an
appropriate multiple of the $\fg$--invariant expression representing
the tractor metric. Using these facts, one easily computes that the
$\frak p_+$--action as a map $\ce_a\otimes \Cal E^{(AB)_0}\to \Cal
E^{(AB)_0}$ is in vector notation given by 
$$
\ph_i\cdot
\begin{pmatrix}
  \si \\ \mu_a \\ A_{ab} \quad | \quad \al\\ \nu_a \\ \rho
\end{pmatrix}=
\begin{pmatrix}
  0\\ -2\si\ph_a \\ -\ph_{(a}\mu_{b)_0}\quad | \quad \ph^i\mu_i \\
  2\ph^iA_{ia}-\tfrac{n+2}{n}\al\ph_a\\ \ph^i\nu_i
\end{pmatrix}. 
$$
>From this, we can determine the formula for the curved Casimir
operator as in \ref{2.4} to obtain
$$\Cal C
\begin{pmatrix}
  \si\\ \mu_a \\ A_{ab}\quad |\quad \al \\ \nu_a \\ \rho 
\end{pmatrix}=
\begin{pmatrix}
\be_0\si \\ \be_1\mu_a+4\nabla_a\si \\
\be_2^1A_{ab}+2\nabla_{(a}\mu_{b)_0}+4\Rho_{(ab)_0}\si \quad |\quad 
\be_2^2\al-2\nabla^c\mu_c-4\Rho\si\\
\be_3\nu_a-4\nabla^cA_{ca}-4\Rho^c{}_{(c}\mu_{a)_0}+
2\tfrac{n+2}{n}\nabla_a\al-2\tfrac{n+2}{n}\Rho_a{}^c\mu_c\\
\be_4\rho-2\nabla_c\nu^c+4\Rho^{cd}A_{cd}-2\tfrac{n+2}{n}\Rho\al 
\end{pmatrix}
$$
Computing the Casimir eigenvalues corresponding to the irreducible
components which occur in that formula is straightforward and gives
$$
\begin{pmatrix}
  \be_0 \\ \be_1 \\ \be_2^1 \quad |\quad \be_2^2 \\ \be_3\\ \be_4
\end{pmatrix}=
\begin{pmatrix}
  w(w+n)+4w+2n+4 \\ w(w+n)+2w+2n \\ w(w+n)+2n \quad | \quad w(w+n)\\
w(w+n)-2w \\ w(w+n)-4w-2n+4
\end{pmatrix}.
$$
The differences of $\be_0$ from these numbers are given by
\begin{equation}
  \label{diff}
\begin{pmatrix}
  0 \\ 2w+4 \\ 4w+4 \quad | \quad 4w+2n+4 \\ 6w+2n+4 \\ 8w+4n
\end{pmatrix}  
\end{equation}
The critical weight for which we can expect an operator from the
top slot to the bottom slot is therefore given by $w=-m$ in dimension
$n=2m$. Inserting this into \eqref{diff}, we obtain
\begin{equation}
  \label{crit-diff}
\begin{pmatrix}
  0 \\ 4-n \\ 4-2n \quad | \quad 4 \\ 4-n \\ 0
\end{pmatrix}.  
\end{equation}
This already shows that something special will happen in dimension
four, since there we obtain a coincidence of four (rather than two) of
the Casimir eigenvalues. There would be another potential speciality
(a coincidence of three of the eigenvalues) in dimension $n=2$, but
this is not geometrically relevant.

According to \ref{2.2a}, an operator from the top slot to the bottom
slot is induced by $(\Cal C-\be_4)\o (\Cal C-\be_3)\o(\Cal
C-\be_2^1)\o(\Cal C-\be_2^2)\o (\Cal C-\be_1)$. To compute the
principal part of this induced operator, one can apply this
composition to an element for which only the top component is nonzero.
Moreover, observe that any derivative moves down one level, so terms
in lower levels which contain only few derivatives can be ignored.
Finally, one can freely commute derivatives when determining the
principal part. Using this simplifications and computing the
composition in the opposite order as written above, it is easy to
verify directly that up to a nonzero factor, the principal part equals
$(n-4)\Delta^2\si$. In particular, for $n\neq 4$ the principal part is
nonzero and we have constructed a conformally invariant square of the
Laplacian.

\subsection{The square of the Laplacian in dimension 4}\label{3.2}
In dimension four, the operator considered in \ref{3.1} reads as
$(\Cal C-\be_4)^3\o (\Cal C-\be_2^1)\o(\Cal C-\be_2^2)$ because of the
additional coincidences of eigenvalues. From \ref{3.1} we see that the
(fourth order) principal part of the induced operator $\ce\to\ce[-4]$
vanishes, and indeed we shall see from the further discussion, that
this operator is identically zero. Still we can obtain a conformally
invariant square of the Laplacian in dimension four from curved
Casimirs. Namely, we will show that actually the operator $(\Cal
C-\be_4)^2\o (\Cal C-\be_2^1)\o(\Cal C-\be_2^2)$ induces such a
square, but this needs some verifications.

Indeed, let us write the natural filtration of the bundle $\Cal
T=\ce^{(AB)_0}[w]$ as $\Cal T=\Cal T^0\supset\Cal
T^1\supset\dots\supset\Cal T^4\supset \{0\}$. Now by construction,
$(\Cal C-\be_2^1)\o(\Cal C-\be_2^2)$ maps sections of $\Cal T^2$ to
sections of $\Cal T^3$, and each occurrence of $\Cal C-\be_4$ maps
sections of $\Cal T$ to sections of $\Cal T^1$, sections of $\Cal T^1$
to sections of $\Cal T^2$, sections of $\Cal T^3$ to sections of $\Cal
T^4$, and sections of $\Cal T^4$ to zero. Thus the composition $(\Cal
C-\be_4)^2\o (\Cal C-\be_2^1)\o(\Cal C-\be_2^2)$ vanishes on $\Ga(\Cal
T^2)$, maps $\Ga(\Cal T^1)$ to $\Ga(\Cal T^4)$ and all of $\Ga(\Cal
T)$ to $\Ga(\Cal T^3)$. In particular, it induces operators
\begin{gather*}
  \Ga(\ce)=\Ga(\Cal T/\Cal T^1)\to\Ga(\Cal T^3/\Cal T^4)=\Ga(\ce_a[-2])\\
  \Ga(\ce_a)=\Ga(\Cal T^1/\Cal T^2)\to\Ga(\Cal T^4)=\Ga(\ce[-4]).
\end{gather*}
If we can prove that both these operators vanish, then we  get
an induced operator $\Ga(\ce)\to\Ga(\ce[-4])$ as required. Since this
is induced by a composition of four curved Casimirs, it follows
immediately that the symbol is induced by the four--fold action of
$\frak p_+$ and hence we have found an invariant square of the
Laplacian. 

It turns out that we can write the two operators whose vanishing we
want to prove as compositions. Since $\be_0=\be_1=\be_3=\be_4$, the
operator $\Cal C-\be_4$ induces invariant operators $\Ga(\Cal T/\Cal
T^1)\to\Ga(\Cal T^1/\Cal T^2)$ as well as $\Ga(\Cal T^3/\Cal
T^4)\to\Ga(\Cal T^4)$, and these are just the exterior derivative $d$
mapping functions to 1--forms, respectively the divergence $\delta$,
which is a formal adjoint to this.   On
the other hand, the composition $(\Cal C-\be_1)\o(\Cal
C-\be_2^1)\o(\Cal C-\be_2^2)$ induces an invariant operator
$T:\Ga(\Cal T^1/\Cal T^2)\to\Ga(\Cal T^3/\Cal T^4)$, so this maps
1--forms to 3--forms. The two operators we have to study are the
compositions $T\o d$ and $\delta\o T$, so we have to prove that these
vanish. We do this by showing that  $T$ is the Maxwell
operator (as expected).

Using the formula for $\Cal C$ from \ref{3.1}, a simple direct
computation shows that the operator $T$ maps $\mu_a$ to
$$
-4\nabla^c\nabla_{(c}\mu_{a)_0}+3\nabla_a\nabla^c\mu_c+
8\Rho^c{}_{(c}\mu_{a)_0}+6\Rho_a{}^c\mu_c.
$$
Now expanding the definition of the tracefree symmetric part
respectively of the Rho--tensor immediately leads to the identities
\begin{gather*}
  -4\nabla^c\nabla_{(c}\mu_{a)_0}=-2\nabla^c\nabla_c\mu_a-
  2\nabla^c\nabla_a\mu_c+\nabla_a\nabla^c\mu_c\\
  8\Rho^c{}_{(c}\mu_{a)_0}=4\Rho\mu_a+2\Rho_a{}^c\mu_c\\
  \nabla_a\nabla^c\mu_c=\nabla^c\nabla_a\mu_c-2\Rho_a{}^c\mu_c-\Rho\mu_a.
\end{gather*}
Putting this together, we immediately get
$T(\mu_a)=2\nabla^c\nabla_{[a}\mu_{c]}$ and this completes the
argument.

While we do not intend to discuss the concept of strong invariance in
detail in this paper, we want to make a brief comment on these issues.
The curved Casimir operators themselves are of course strongly
invariant in every sense, since they are of first order. Consequently,
any operator directly induced by a polynomial in curved Casimirs is
strongly invariant, too. In particular, the construction of \ref{3.1}
provides strongly invariant squares of the Laplacian in dimensions
different from $4$. The construction in dimension four however depends
on vanishing of the compositions $T\o d$ and $\delta\o T$, which (like
the equation $d\o d=0$) are not valid in a strong sense. Hence in
dimension 4 we cannot conclude that we get a strongly invariant
operator.

\subsection{The cube of the Laplacian}\label{3.3}
To conclude this article, we briefly outline what happens for the cube
of the Laplacian. The relevant bundle  to obtain a cube of the
Laplacian is of course $S^3_0\Cal E^A$, which has composition series 
\begin{multline*}
\textstyle\ce[w+3]\lpl \ce_a[w+3]\lpl
\binom{\ce_{(ab)_0}[w+3]}{\ce[w+1]}\lpl
\binom{\ce_{(abc)_0}[w+3]}{\ce_a[w+1]}\lpl\\
\textstyle \binom{\ce_{(ab)_0}[w+1]}{\ce[w-1]}
\lpl\ce_a[w-1]\lpl\ce[w-3]
  \end{multline*}
We use a vector notation similar as before. Computing the Casimir
eigenvalues is straightforward, and shows that the weight for which
one may expect an operator from the top slot to the bottom slot is
again $w=\frac{-n}{2}$. For this the differences of the Casimir
eigenvalue for the top slot from the other Casimir eigenvalues form
the pattern 
$$
\begin{pmatrix}
  0 \\ 6-n \\ 2(4-n) \quad |\quad 8\\
6-3n \quad |\quad 10-n\\ 2(4-n) \quad |\quad 8\\
6-n \\ 0
\end{pmatrix},
$$
which shows that additional coincidences of Casimir eigenvalues occur
in dimensions $4$, $6$, and $10$. While the special role of dimensions
$4$ (for which non--existence of a conformally invariant power of the
Laplacian is proved in \cite{Graham:nonex}) and $6$ (for which the
cube is the critical power of the Laplacian) has to be expected, the
special role of dimension $10$ comes as a surprise. 

To compute the curved Casimir, the main input is again the action of
$\frak g_1$ which, viewed as a map $\ce_a\otimes S^3_0\ce^A\to
S^3_0\ce^A$, is given by 
$$
\ph_i\cdot
\begin{pmatrix}
  \si \\ \mu^a \\ A_{ab} \quad |\quad \al \\ \Ph_{abc} \quad |\quad
  \nu^a\\ B_{ab} \quad |\quad \be\\ \tau^a\\ \rho
\end{pmatrix}=
\begin{pmatrix}
  0\\ -3\ph_a\\ -2\ph_{(a}\mu_{b)_0} \quad |\quad \ph^i\mu_i\\
-\ph_{(a}A_{bc)_0} \quad |\quad -2\tfrac{n+2}n\al\ph_a+2\ph^iA_{ia}\\ 
-\tfrac{n+4}{n+2}\ph_{(a}\nu_{b)_0}+3\ph^i\Ph_{iab} \quad |\quad
\ph^i\nu_i\\ -\tfrac{n+4}{n}\be\ph_a+2\ph^iB_{ia}\\\ph^i\tau_i
\end{pmatrix}.
$$
>From this, one easily derives the full formula for the curved
Casimir operator on the bundle $S^3_0\ce^A[w]$. According to
\ref{2.2a}, the operator to consider is
\begin{equation}
  \label{cubecomp}
(\Cal C-\be_0)\o(\Cal C-\be_1)^2\o(\Cal C-\be_2^1)^2\o(\Cal
C-\be_2^2)^2\o(\Cal C-\be_3^1)\o(\Cal C-\be_3^2),   
\end{equation}
where the squares are due to the fact that $\be_5=\be_1$ and
$\be_4^i=\be_2^i$ for $i=1,2$. To compute the principal part of the
induced operator, one proceeds in a manner similar to \ref{3.1} above.
That is by working through the composition starting with the factor
$\Cal C-\be_0$ and then working down level by level. One takes only
terms of high enough order in each level, and freely commutes
derivatives. This shows that, up to a nonzero factor, the principal
part is given by
$$
\si\mapsto (n-4)(n-6)(n-10)\Delta^3\si.
$$
We want to point out however, that while the factors $(n-4)$,
$(n-6)$, and $(n-10)$ occur as differences of Casimir eigenvalues, the
fact that they arise in the principal part is not at all
straightforward, but has to be verified by rather nasty computations.
In all dimensions except for these three critical ones, our operator
directly defines a conformally invariant cube of the Laplacian.

Concerning the critical dimensions, the situation is the following. The
easiest of these cases is dimension $10$. Here there is an additional
coincidence of Casimir eigenvalues, since $\be_3^2=\be_0$. Let us
 write $\Cal T=S^3_0\ce^A$ and  us denote the canonical
filtration of $\Cal T$ by $\Cal T=\Cal T^0\supset\dots\supset\Cal
T^6\supset\{0\}$. Now consider the composition
$$
(\Cal C-\be_3^2)\o(\Cal C-\be_2^1)\o(\Cal C-\be_2^2)\o(\Cal
C-\be_1).
$$
This maps $\Ga(\Cal T)$ to $\Ga(\Cal T^3)$, and if we project to
$\Cal T^3/\Cal T^4$ and then further to the component $\ce_a[-4]$
(which corresponds to the eigenvalue $\be_3^2$), then the composition
vanishes on $\Ga(\Cal T^1)$. Hence it induces an operator from
sections of $\Cal T/\Cal T^1\cong\ce[-2]$ to sections of $\ce_a[-4]$.
(It is known from the classification of conformally invariant
operators, that this has to vanish in the conformally flat case.) Now
a direct computation shows that this operator actually is always
identically zero. This shows that
$$
(\Cal C-\be_3^1)\o(\Cal C-\be_3^2)\o(\Cal C-\be_2^1)\o(\Cal
C-\be_2^2)\o(\Cal C-\be_1)
$$
maps all of $\Ga(\Cal T)$ to $\Ga(\Cal T^4)$. Hence if we further
apply $(\Cal C-\be_5)\o (\Cal C-\be_4^1)\o(\Cal C-\be_4^2)$, the
result maps all of $\Ga(\Cal T)$ to $\Ga(\Cal T^6)$.

Similarly, we can consider the composition 
$$
(\Cal C-\be_5)\o (\Cal C-\be_4^1)\o(\Cal C-\be_4^2)\o(\Cal C-\be_3^2) 
$$
on the space of those sections of $\Cal T^3$ whose image in $\Cal
T^3/\Cal T^4$ is a section of the component $\ce_a[-4]$ only.  As
before, this clearly maps all such sections to sections of $\Cal T^6$,
and since $\be_3^2=\be_6$ it vanishes on sections of the subbundle
$\Cal T^4$. Hence we get an induced operator from sections of
$\ce_a[-4]$ to sections of $\Cal T^6=\ce[-8]$. Once again, a direct
computation shows that this operator vanishes identically (which in
the conformally flat case follows from the known classification
results). Now on the other hand, the composition
$$
(\Cal C-\be_3^1)\o(\Cal C-\be_2^1)\o(\Cal C-\be_2^2)\o (\Cal C-\be_1)
$$
maps $\Ga(\Cal T^1)$ to $\Ga(\Cal T^3)$ and projecting to $\Cal
T^3/\Cal T^4$ the result lies in $\Ga(\ce_a[-4])$ only. Together with
the above observation we conclude that if in the composition
\eqref{cubecomp} we leave out one of the two factors $(\Cal C-\be_0)$,
then the result still maps sections of $\Cal T$ to sections of $\Cal
T^6$ and vanishes on sections of $\Cal T^1$. Hence we again get an
induced operator mapping sections of $\Cal T/\Cal T^1\cong\ce[-2]$ to
sections of $\ce[-8]\cong\Cal T^6$. Of course, this also implies that
the original composition \eqref{cubecomp} induces the zero operator in
dimension $10$.

A similar computation as for general dimensions now shows that the
principal part of this operator is a nonzero multiple of
$\si\mapsto\Delta^3\si$. Hence we have obtained a cube of the
Laplacian in dimension $10$, although we cannot conclude that this is
strongly invariant.

Next, let us discuss dimension $n=4$, for which there is no
conformally invariant cube of the Laplacian by
\cite{Graham:nonex}. Due to the coincidences of Casimir eigenvalues,
the composition \eqref{cubecomp} here specialises to 
\begin{equation}
  \label{cubecomp4}
(\Cal C-\be_0)^3\o(\Cal C-\be_1)^2\o(\Cal C-\be_2^2)^2\o(\Cal
C-\be_3^1)\o(\Cal C-\be_3^2).   
\end{equation}
One might hope that one can define a cube of the Laplacian in
dimension four, at last for a certain class of conformal manifolds by
leaving out one of the three factors $(\Cal C-\be_0)$. This turns out
to work however, only on the subcategory of locally conformally flat
structures. 

The pattern is similar to that arising for the square of the Laplacian
in dimension four. The composition $(\Cal C-\be_0)\o(\Cal
C-\be_3^1)\o(\Cal C-\be_3^2)$ is easily seen to induce a second order
operator $\Ph$ mapping sections of $\Cal E_{(ab)_0}[1]\subset\Cal
T^2/\Cal T^3$ to sections of $\Cal E_{(ab)_0}[-1]\subset\Cal T^4/\Cal
T^5$. Likewise, the composition $(\Cal C-\be_0)\o(\Cal C-\be_1)$
induces an operator $\Ps_1$ mapping sections of $\ce[1]\cong\Cal
T/\Cal T^1$ to sections of $\Cal E_{(ab)_0}[1]\subset\Cal T^2/\Cal
T^3$ as well as an operator $\Ps_2$, which maps sections of $\Cal
E_{(ab)_0}[-1]\subset\Cal T^4/\Cal T^5$ to sections of
$\ce[-5]\cong\Cal T^6$. To get and induced operator
$\Ga(\ce[1])\to\Ga(\ce[-5])$ after leaving out one of the three
factors $(\Cal C-\be_0)$ in \eqref{cubecomp4}, one needs the
compositions $\Ph\o\Ps_1$ and $\Ps_2\o\Ph$ to vanish identically.
However, it turns out that both these compositions actually are second
order operators with Weyl curvature in the principal symbol and a
tensorial part involving the Bach tensor. Further, from the explicit
form for the principal symbol one may see that 
it 
 vanishes only in
the locally flat case (where this also follows from the classification
results). In the latter case, one can then compute the principal part
similarly as before to see that one indeed does obtain a conformally
invariant cube of the Laplacian on locally conformally flat
$4$--manifolds, but not for a larger class.

Finally, in the critical dimension $n=6$ some details remain
unresolved. Due to the coincidences of Casimir eigenvalues, the
composition \eqref{cubecomp} specialises to
\begin{equation}
  \label{cubecomp6}
(\Cal C-\be_0)^3\o(\Cal C-\be_2^1)^2\o(\Cal
C-\be_2^2)^2\o(\Cal C-\be_3^1)\o(\Cal C-\be_3^2).   
\end{equation}
As for the square of the Laplacian in dimension four, the hope would
be to leave out one of the three factors $(\Cal C-\be_0)$ and still
get an induced operator. Also, the verifications to be made are
analogous to ones from \ref{3.2}. The composition
$$
(\Cal C-\be_0)\o(\Cal C-\be_2^1)^2\o(\Cal
C-\be_2^2)^2\o(\Cal C-\be_3^1)\o(\Cal C-\be_3^2)
$$
induces a fourth order operator $T:\Ga(\ce_a)\to\Ga(\ce_a[-4])$. On
the other hand, $(\Cal C-\be_0)$ induces the exterior derivative
$d:\Ga(\ce)\to\Ga(\ce_a)$ as well as the divergence
$\delta:\Ga(\ce_a[-4])\to\Ga(\ce[-6])$. Leaving out one of the three
factors $(\Cal C-\be_0)$ in \eqref{cubecomp6}, the result induces an
operator $\Ga(\ce)\to \Ga(\ce[-6])$ if and only if the compositions
$T\o d$ and $\delta\o T$ vanish identically. Of course, this is true
in the flat case, so there the construction again works. While we have
been able to compute a complete formula for $T$ in the curved case,
computing the two compositions explicitly seems to be a serious 
task. To sort out this problem new ideas would be helpful.

\end{document}